\documentclass[a4,11pt]{article}
\usepackage{amsfonts}
\usepackage{amsmath}
\usepackage{amssymb}

\usepackage{amsthm}

\theoremstyle{plain}
 \newtheorem{theorem}{Theorem}[section]
 
 \newtheorem*{theorem*}{Theorem}
 \newtheorem{proposition}[theorem]{Proposition}
 \newtheorem{lemma}[theorem]{Lemma}
 \newtheorem{corollary}[theorem]{Corollary}
 \newtheorem{problem}[theorem]{Problem}
 
\theoremstyle{definition}
 \newtheorem{definition}[theorem]{Definition}
\theoremstyle{remark}
 
 \newtheorem{example}[theorem]{Example}
\numberwithin{equation}{section}

\title{New parameters of subsets in polynomial schemes}

\author{Sho Suda\\
{\small Division of Mathematics, Graduate School of Information Sciences, Tohoku University,} \\
{\small 6-3-09 Aramaki-Aza-Aoba, Aoba-ku, Sendai 980-8579, Japan}}
\date{}
\begin{document}
\maketitle

\begin{abstract}
We define new parameters, a zero interval and a dual zero interval, of subsets in $P$- or $Q$-polynomial schemes.
A zero interval of a subset in a $P$-polynomial scheme is a successive interval index for which the inner distribution vanishes, and 
a dual zero interval of a subset in a $Q$-polynomial scheme is a successive interval index for which the dual inner distribution vanishes.
We derive the bounds of the lengths of a zero interval and a dual zero interval using the degree and dual degree respectively, 
and show that a subset in a $P$-polynomial scheme (resp. a $Q$-polynomial scheme) having a large length of a zero interval (resp. a dual zero interval) induces a completely regular code (resp. a $Q$-polynomial scheme).
Moreover, we consider the spherical analogue of a dual zero interval.  
\end{abstract}

\section{Introduction}
A subset $C$ in a polynomial scheme $(X,\mathcal{R})$ with class $d$ has several important parameters: the minimum distance and the width in the case of a $P$-polynomial scheme, and the strength and the width in a $Q$-polynomial scheme.
In 1973, Delsarte defined and widely studied the minimum distance and strength with the duality of translation schemes in \cite{D}.
Later, in 2003, Brouwer et al. defined the width and dual width in \cite{BGKM}.
The concept of minimum distance (resp. strength) is equivalent to the successive subsequence consisting of $0$ from the first term of the inner distribution (resp. dual inner distribution) of $C$.
On the other hand, the concept of width (resp. the dual width) is the successive subsequence consisting of $0$ to the last term of the inner distribution (resp. dual inner distribution) of $C$.
The length of the successive subsequence consisting of $0$ in each case is bounded above by the degree or dual degree.
If the gap is close to $0$, then $C$ induces a completely regular code or a $Q$-polynomial scheme.

As described in this paper, we define new parameters---the zero interval and dual zero interval---for a subset in a polynomial scheme.
A zero interval is a successive sequence intermediately for which the inner distribution vanishes;
a dual zero interval of a subset in a $Q$-polynomial scheme is a successive sequence intermediately for which the dual inner distribution vanishes.
For the case of a $P$-polynomial scheme (resp. $Q$-polynomial scheme), we show that the length of a zero interval (resp. a dual zero interval) of $C$ is bounded above by twice the value of the dual degree (resp. degree).
Moreover, if the gap is close to $0$, $C$ carries a completely regular code (resp. a $Q$-polynomial scheme) as Delsarte theory.
We give several examples having great length, which relate linear perfect codes.
 
Finally, we consider a spherical analogue of a dual zero interval just as a spherical design.
Using the theory of dual zero interval, we derive the sufficient condition that a finite non-empty subset carries a $Q$-polynomial scheme.
\section{Zero interval of subsets in $P$-polynomial schemes}
First, let $(X,\mathcal{R})$ be a $d$-class $P$-polynomial scheme (or a distance-regular graph with diameter $d$ and path-length distance $\partial$).
Then, for subset $C$ in $X$, we define the characteristic vector $\chi$ as a column vector indexed by $X$ whose $x$-th entry is $1$ if $x\in C$, and $0$ otherwise. We define the inner distribution $\mathbf{a}=(a_0,a_1,\ldots,a_d)$ of $C$ as $a_i=\frac{1}{|C|}\chi^TA_i\chi$ for $i\in \{0,1,\ldots,d\}$.
We define a dual degree set $S^*(C)$ of $C$ as 
\begin{align*}
S^*(C)=\{j \mid  1\leq j\leq d,\chi^TE_j\chi\neq0 \},
\end{align*}
and dual degree $s^*$ as the cardinality of the dual degree set $S^*(C)$.
A polynomial $F(x)$ is called a dual annihilator polynomial of $C$ if 
\begin{align*}
F(\theta_i)=0\ \text{ for any } i\in S^*(C),
\end{align*}
where $\theta_0,\ldots,\theta_d$ are eigenvalues of the $P$-polynomial scheme.
We define an $|X| \times (d+1)$ matrix $B$ as 
$$B_{x,i}=e_x^TA_i\chi,$$
where $e_x$ denote the characteristic vector of $\{x\}$ for $x\in X$.
Additionally, $B$ is called the outer distribution matrix of $C$.
Since $B=(A_0\chi,A_1\chi,\ldots,A_d\chi)$, $BQ=|X|(E_0\chi,E_1\chi,\ldots,E_d\chi)$ holds where 
$Q$ is the second eigenmatrix of $(X,\mathcal{R})$.
Therefore $\mathrm{rank}(B)=s^*+1$.
We define the distance of $x\in X$ to $C$ as $
\partial(x,C)=\min\{
\partial(x,y)\mid y\in C\}$ and the covering radius of $C$ by $\rho=\max\{
\partial(x,C)\mid x\in X\}$. 
Here, $C$ is called completely regular if, for $x\in X$, $B_{x,i}$ depends only on $\partial(x,C)$ and $i$.
We define a parameter of subsets in $P$-polynomial schemes as follows:
\begin{definition}
Let $C$ be a non-empty subset of a $d$-class $P$-polynomial scheme $X$.
For $0\leq w\leq d-1$ and $1\leq t\leq d-w$, $C$ has a zero interval $\{w+1,\ldots,w+t\}$ if $a_{w+1}=\cdots=a_{w+t}=0$.
\end{definition}
If $(X,\mathcal{R})$ is a bipartite $P$-polynomial scheme and $C$ is a halved graph, 
then the inner distribution of $C$ satisfies that $a_i$ is $0$ if $i$ is odd, $k_i$ otherwise, where $k_i$ is the valency of the regular graph $(X,R_i)$.
Therefore a zero interval of $C$ is not necessarily uniquely determined.

For a subset $C$ having a zero interval $\{w+1,\ldots,w+t\}$, without loss of generality, we may assume $a_w\neq 0$ and one of the following:
\begin{enumerate}
\item $w+t+1\leq d$ and $a_{w+t+1}\neq0$,
\item $w+t=d$.
\end{enumerate} 
The zero interval $\{1,\ldots,t\}$  with $a_{t+1}\neq0$ is equivalent to the minimum distance, 
and the zero interval $\{w+1,\ldots,d\}$ with $a_w\neq 0$ is equivalent to the width.

The following proposition gives a relation between a dual annihilator polynomial and the inner distribution of a subset $C$. 
\begin{proposition}\label{P-car}
Let $C$ be a non-empty subset of a $d$-class $P$-polynomial scheme $X$.
Let $F(x)$ be a dual annihilator polynomial of $C$. 
We assume that there exist $\{f_i\}_{i=0}^d$ such that, for any $\theta\in \{\theta_0,\ldots,\theta_d\}$, $F(\theta)=\sum_{k=0}^df_kv_k(\theta)$, where $\{v_i(x)\}_{i=0}^d$ are the orthogonal polynomials corresponding to the first eigenmatrix.
Then $\frac{F(\theta_0)|C|}{|X|}=\sum_{k=0}^df_ka_k$.
\end{proposition}
\begin{proof}
Since $F(\theta_i)=\sum_{k=0}^df_kv_k(\theta_i)$ for any $1\leq i\leq d$,  
\begin{align*}
\sum_{i=0}^dF(\theta_i)E_i&=\sum_{i=0}^d(\sum_{k=0}^df_kv_k(\theta_i))E_i \displaybreak[0]\\
&=\sum_{k=0}^df_k\sum_{i=0}^dv_k(\theta_i)E_i \displaybreak[0]\\
&=\sum_{k=0}^df_kA_k.
\end{align*} 
Therefore $\sum_{i=0}^dF(\theta_i)\chi^TE_i\chi=\sum_{k=0}^df_k\chi^TA_k\chi$ holds.
Since $F(x)$ is a dual annihilator polynomial of $C$, $\chi^TE_0\chi=\frac{|C|^2}{|X|}$ and $\chi^TA_k\chi=|C|a_k$, we have the desired result.
\end{proof}
The following proposition holds, as in the case of minimum distance and width of codes.
Proposition~\ref{P-bound} (2) is already obtained in \cite{BGKM}, but we give the another proof.
\begin{proposition}\label{P-bound}
Let $C$ be a non-empty subset of a $d$-class $P$-polynomial scheme $X$ having a zero interval $\{w+1,\ldots,w+t\}$ with $a_w\neq 0$ and dual degree $s^*$.
\begin{enumerate}
\item If $w+t+1\leq d$ and $a_{w+t+1}\neq0$, then $t\leq 2s^*$ holds.  
\item If $w+t=d$, then $t\leq s^*$ holds.
\end{enumerate}
\end{proposition} 
\begin{proof}
In each case, if $d-w\leq s^*$ holds, then $t\leq d-w\leq s^*$ holds.
Therefore we assume that $s^*+1\leq d-w$.

We define $F(x)=\prod_{i\in S^*(C)}\frac{x-\theta_i}{\theta_0-\theta_i}$ and $G(x)=v_{w+s^*+1}(x)F(x)$.
Here, $F(x)$ and $G(x)$ are dual annihilator polynomials of $C$.
We define $\{f_j\}_{j=0}^{s^*}$ as the coefficients of $F(x)$  in terms of $\{v_i(x)\}_{i=0}^d$ i.e., $$F(x)=\sum_{j=0}^{s^*}f_jv_j(x).$$
Then, for any $\theta\in \{\theta_0,\ldots,\theta_d\}$,
\begin{align*} 
G(\theta)&=v_{w+s^*+1}(\theta)F(\theta)\displaybreak[0]\\
&=\sum\limits_{j=0}^{s^*}f_jv_{w+s^*+1}(\theta)v_j(\theta) \displaybreak[0]\\
&=\sum\limits_{j=0}^{s^*}f_j\sum\limits_{k=w+s^*+1-j}^{\min\{d,w+s^*+1+j\}}p_{w+s^*+1,j}^kv_k(\theta) \displaybreak[0]\\
&=\sum\limits_{k=w+1}^{\min\{d,w+2s^*+1\}}\Bigl(\sum\limits_{j=|w+s^*+1-k|}^{s^*}f_jp_{w+s^*+1,j}^k\Bigr) v_k(\theta).
\end{align*}
Set $g_k=0$ for $0\leq k\leq w$ and $g_k=\sum\limits_{j=|w+s^*+1-k|}^{s^*}f_jp_{w+s^*+1,j}^k$ for $w+1\leq k\leq \min\{d,w+2s^*+1\}$, 
$G(x)$ is satisfied the assumption of Proposition~\ref{P-car}. 
Then by Proposition~\ref{P-car}, we obtain 
$$\frac{G(\theta_0)|C|}{|X|}=\sum\limits_{k=w+1}^{\min\{d,w+2s^*+1\}}g_ka_k.$$ 

In the case of (1); If $t\geq 2s^*+1$ holds, then $a_k=0$ for $w+1\leq k\leq w+2s^*+1$. 
Since $G(\theta_0)=v_{w+s^*+1}(\theta_0)>0$, $|C|=0$, which is a contradiction.
Consequently, $t\leq 2s^*$ holds.

In the case of (2); Then $a_k=0$ for $w+1\leq k\leq d$. 
Since $G(\theta_0)=v_{w+s^*+1}(\theta_0)>0$, $|C|=0$, which is a contradiction.
Consequently, $t\leq s^*$ holds.
\end{proof}
The following is an algebraic proof of Proposition~\ref{P-bound} using the Terwilliger algebra.
\begin{proof}
(Second proof of Proposition~\ref{P-bound})
We denote $V=\mathbb{C}^{|X|}$.
Fix a base point $x\in C$. Let $E_i^*=E_i^*(x)$ be the diagonal matrix with $E_i^*(y,y)=A_i(x,y)$ for each $i\in \{0,1,\ldots,d\}$.
Clearly $E_i^*V\cap E_j^*V=0$ for distinct integers $i,j\in \{0,1,\ldots,d\}$.

(1); 
If $w+t+1\leq d$, $a_{w+t+1}\neq0$ hold,
then there exist uniquely integers $w_x,t_x$ such that 
$0\leq w_x\leq w$, $1\leq t\leq w_x+t_x-w$ and
$e_x^TB$ vanishes for a successive indices $\{w_x+1,\ldots,w_x+t_x\}$ and does not vanish for indices $w_x, w_x+t_x+1$.
In particular $t\leq t_x$ holds.
Therefore $\chi|_{E_{w_x}^*V}\neq0 $, $\chi|_{E_{w_x+1}^*V}=\cdots=\chi|_{E_{w_x+t_x}^*V}=0$ and $\chi|_{E_{w_x+t_x+1}^*V}\neq0$ hold. 
Hence $A_1^i\chi|_{E_{w_x+i}^*V}\neq0 $, $A_1^i\chi|_{E_{w_x+i+1}^*V}=\cdots=A_1^i\chi|_{E_{w_x+t_x-i}^*V}=0$ and $A_1^i\chi|_{E_{w_x+t_x-i+1}^*V}\neq0$ hold for each $i$.
Therefore, since $w_x+i< w_x+t_x-i+1$ holds if and only if $i\leq \lfloor t_x/2\rfloor$ holds,  
$\chi,A_1\chi,\ldots,A_1^{\lfloor t_x/2\rfloor}\chi$ are linearly independent.
Therefore $\lfloor t_x/2\rfloor+1\leq \mathrm{rank}(B)=s^*+1$, that is, $t_x\leq 2s^*$ holds.
Since $t\leq t_x$, the assertion holds.

(2); 
If $w+t=d$ holds,
then there exist uniquely integer $w_x$ such that 
$0\leq w_x\leq w$ and 
$e_x^TB$ vanishes for a successive indices $\{w_x+1,\ldots,d\}$ and does not vanish for index $w_x$.
Define $t_x=d-w_x$.
Since $w_x\leq w$ holds, $t\leq t_x$ holds.
Therefore $\chi|_{E_{w_x}^*V}\neq0 $, $\chi|_{E_{w_x+1}^*V}=\cdots=\chi|_{E_{d}^*V}=0$
Hence $A_1^i\chi|_{E_{w_x+i}^*V}\neq0 $, $A_1^i\chi|_{E_{w_x+i+1}^*V}=\cdots=A_1^i\chi|_{E_{d}^*V}=0$  for each $i$.
Therefore, 
$\chi,A_1\chi,\ldots,A_1^{t_x}\chi$ are linearly independent.
Therefore $t_x+1\leq \mathrm{rank}(B)=s^*+1$, that is, $t_x\leq s^*$ holds.
Since $t\leq t_x$, the assertion holds.
\end{proof}
The following is the main theorem in this section.
The proof is a slight generalization of that of \cite[Theorem~1]{BGKM}. 
\begin{theorem}\label{P-mainth}
Let $C$ be a non-empty subset of a $d$-class $P$-polynomial scheme $X$ having dual degree $s^*$ and $B$ be the outer distribution matrix of $C$.
Assume that there exists $w\in \{0,1,\ldots,d-s^*\}$ such that $a_w>0$ and for any $i\in\{0,1,\ldots,s^*\}$ and $x\in X$ with $\partial(x,C)=i$
$$\quad B_{x,j}=0 \text{ if } w+i+1\leq j\leq w+s^*.$$
Then $C$ is completely regular.
\end{theorem}
\begin{proof}
By the assumption $a_w>0$, there exist $y,z\in C$ such that $\partial(y,z)=w$.
We take $z_i\in X$ with $\partial(z_i,z)=i$ and $\partial(z_i,y)=w+i$ for $0\leq i\leq s^*$.
Since $\partial(z_i,z)=i$, $\partial(z_i,C)\leq i$ holds.
And since $B_{z_i,w+i}>0$ and the assumption, $\partial(z_i,C)\geq i$ holds.
Consequently, $\partial(z_i,C)=i$ holds for each $i$. 

Let $x\in X$ be a vertex with $\partial(x,C)=l$.
The submatrix $M$ of $B$ obtained by restricting the row to $\{z_0,\ldots,z_{s^*}\}$ satisfies $M_{z_i,i}>0$ and $M_{z_i,j}=0$ for $j<i$,
which shows that $\mathrm{rank}(M)=s^*+1=\mathrm{rank}(B)$ holds; 
consequently, $\mathrm{rowsp}(M)=\mathrm{rowsp}(B)$ holds.
The row indexed by $x$ is a linear combination of the rows indexed by $z_0,\ldots,z_{s^*}$.
The coefficient of $z_i$ is $0$ for $i<l$ since $B_{x,i}=0$ and $B_{z_i,i}>0$.
Additionally, the coefficient of $z_i$ is $0$ for $l<i\leq s^*$ since $B_{x,j}=0$ for $w+l+1\leq j\leq w+s^*$ by the assumption and $B_{z_i,j}>0$ for $w+l<j\leq w+s^*$.
Since any row sum is $|C|$, the coefficient of $z_l$ is $1$.
For that reason, the row indexed by $x$ is equal to the row indexed by $z_l$, which is independent of the choice of $x$ with $\partial(x,C)=l$.
Therefore $C$ is completely regular.
\end{proof}

\begin{corollary}{\upshape \cite[Theorem~1]{BGKM}}
Let $C$ be a non-empty subset of a $d$-class $P$-polynomial scheme $X$ having width $w$ and dual degree $s^*$.
If $w+s^*=d$, then $C$ is completely regular.
\end{corollary}
\begin{proof}
It is sufficient to verify that the assumption of Theorem~\ref{P-mainth} holds for width $w$.
By the definition of width, $a_w>0$ holds. 
Let $x\in X$ be a vertex with $\partial(x,C)=i$ and $y\in C$ be a vertex with $\partial(x,y)=i$.
For $z\in C$, by triangle equality, 
$w\geq \partial(y,z)\geq|\partial(x,y)-\partial(x,z)|=\partial(x,z)-i$.
Consequently, $\partial(x,z)\leq w+i$ holds. 
Therefore, $B_{x,j}=0 \text{ if } w+i+1\leq j\leq w+s^*$.
\end{proof}

\begin{corollary}\label{P-int}
Let $C$ be a non-empty subset of a $d$-class $P$-polynomial scheme $X$ having a zero interval $\{w+1,\ldots,w+t\}$ and dual degree $s^*$.
If $2s^*-1\leq t$, then $C$ is completely regular.
\end{corollary}
\begin{proof}
It is sufficient to verify that the assumption of Theorem~\ref{P-mainth} holds for $w$ which appears in a zero interval.
By the definition of the zero interval, $a_w>0$ holds.
Let $x\in X$ be a vertex with $\partial(x,C)=i$ and $y\in C$ be a vertex with $\partial(x,y)=i$.
For $z\in C$, by the triangle equality,
When $\partial(y,z)\geq w+t+1$, $w+t+1\leq \partial(y,z)\leq\partial(x,y)+\partial(x,z)\leq i+\partial(x,z)$ holds, and
when $\partial(y,z)\leq w$, $w\geq \partial(y,z)\geq|\partial(x,y)-\partial(x,z)|\geq \partial(x,z)-i$ holds.
Therefore $\partial(x,z)\leq w+i$ or $\partial(x,z)\geq w+2s^*-i$ holds, 
which means, in particular, that $B_{x,j}=0 \text{ if } w+i+1\leq j\leq w+s^*$.
\end{proof}
Corollary~\ref{P-int} gives a sufficient condition that a subset having a nice zero interval carries a completely regular code.
When $w=0$, Corollary~\ref{P-int} implies that a code with minimum distance $t+1$ and dual degree $s^*$ satisfying $2s^*\leq t+1$ carries a completely regular code.
However, in fact, it is well-known in \cite[Theorem~5.13]{D} that the assumption $2s^*-1\leq t+1$ yields the same result, i.e. a code with minimum distance $\delta$ and dual degree $s^*$ satisfying $2s^*-1\leq t+1$ carries a completely regular code.

\begin{example}\label{P-example}
\begin{enumerate}
\item Let $C=\{0,1\} \times \{0\ldots0,1\ldots1\}$ in the binary Hamming scheme $H(2n,2)$.
Then $C$ has a zero interval $\{2,3,\ldots,2n-2\}$ and dual degree set $S^*(C)=\{2k\mid 2\leq k \leq n-1\}$.
Therefore $C$ satisfies $t= 2s^*$, which implies that $C$ is completely regular.
\item Let $C$ be the $[2^m-1,2^m-m,3]$ Hamming code in the binary Hamming scheme $H(2^m,2)$.
Then $C$ has a zero interval $\{w+1,\ldots,w+t\}$ with $w=2^m-4$ and $t=2$ and the dual degree $s^*=1$.
Therefore $C$ satisfies $t=2s^*$, which implies that $C$ is completely regular.
Noting that the minimum distance of $C$ is $3$, this is already shown in \cite[Theorem~5.13]{D}.
\item Let $C$ be the $[2^m,2^m-m,3]$ extended Hamming code in the binary Hamming scheme $H(2^m,2)$.
Then $C$ has a zero interval $\{w+1,\ldots,w+t\}$ with $w=2^m-4$ and $t=3$ and the dual degree $s^*=2$.
For that reason, $C$ satisfies $t=2s^*-1$, which implies that $C$ is completely regular.
Noting that the minimum distance of $C$ is $4$, this is already shown in \cite[Theorem~5.13]{D}.
\item Let $C$ be the $[23,12,7]$ Golay code in the binary Hamming scheme $H(23,2)$.
The inner distribution and the dual inner distribution of $C$ are
\begin{align*}
\mathbf{a}&=(1, 0, 0, 0, 0, 0, 0, 253, 506, 0, 0, 1288, 1288, 0, 0, 506, 253, 0,0,0,0, 0, 0, 1),\\
\mathbf{b}&=4096(1, 0, 0, 0, 0, 0, 0, 0, 506, 0, 0, 0, 1288, 0, 0, 0, 253, 0, 0, 0, 0, 0, 0, 0).
\end{align*} 
Then $C$ has a zero interval $\{w+1,\ldots,w+t\}$ with $w=16$ and $t=6$ and the dual degree $s^*=3$.
Consequently, $C$ satisfies $t=2s^*$, which implies that $C$ is completely regular.
The minimum distance of $C$ is $7$: this is already shown in \cite[Theorem~5.13]{D}.
\item Let $C$ be the $[23,12,7]$ Golay code in the binary Hamming scheme $H(23,2)$ and
 $C_1$ be $C\times \{0,1\}$ in the binary Hamming scheme $H(24,2)$.
The inner distribution and the dual inner distribution of $C_1$ are
\begin{align*}
\mathbf{a}&=(1, 1, 0, 0, 0, 0, 0, 253, 759, 506, 0, 1288, 2576, 1288, 0, 506, 759, 253, 0, 0, 0, 0, 0, 1, 1),\\
\mathbf{b}&=8192(1, 0, 0, 0, 0, 0, 0, 0, 506, 0, 0, 0, 1288, 0, 0, 0, 253, 0, 0, 0, 0, 0, 0, 0, 0).
\end{align*} 
Then $C_1$ has intervals $\{w+1,\ldots,w+t\}$ with $(w,t)=(1,5)$ or $(w,t)=(17,5)$ and the dual degree $s^*=3$.
Consequently, $C$ satisfies $t=2s^*-1$, which implies that $C$ is completely regular.
\item Let $C$ be the $[24,12,8]$ Golay code in the binary Hamming scheme $H(24,2)$.
The inner distribution and the dual inner distribution of $C$ are
\begin{align*}
\mathbf{a}=\frac{1}{4096}\mathbf{b}=(1, 0, 0, 0, 0, 0, 0, 0, 759, 0, 0, 0, 2576, 0, 0, 0, 759, 0, 0, 0, 0, 0, 0, 0, 1).
\end{align*} 
Then $C$ has a zero interval $\{w+1,\ldots,w+t\}$ with $w=16$ and $t=7$ and the dual degree $s^*=4$.
For that reason, $C$ satisfies $t=2s^*-1$, which implies that $C$ is completely regular.
The minimum distance of $C$ is $8$. Therefore, this is already shown in \cite[Theorem~5.13]{D}.
\end{enumerate}
\end{example}
\section{Dual zero interval of subsets in $Q$-polynomial schemes}
Let $(X,\mathcal{R})$ be a $d$ class $Q$-polynomial scheme.
For subset $C$ in $X$, let $\chi$ be the characteristic vector of $C$ and
we define the dual inner distribution $\mathbf{b}=(b_0,b_1,\ldots,b_d)$ of $C$ by $b_i=\frac{|X|}{|C|}\chi^TE_i\chi$ for $i\in \{0,1,\ldots,d\}$.
We define a degree set $S(C)$ of $C$ as 
\begin{align*}
S(C)=\{j \mid 1\leq j\leq d, \chi^TA_j\chi\neq0 \},
\end{align*}
and degree $s$ by the cardinality of the degree set $S(C)$.
A polynomial $F(x)$ is called an annihilator polynomial of $C$ if 
\begin{align*}
F(\theta_i^*)=0\ \text{ for any } i\in S(C),
\end{align*}
where $\theta_0^*,\ldots,\theta_d^*$ are dual eigenvalues of the $Q$-polynomial scheme.
Let $\Delta_C$ denote the diagonal matrix where $\Delta_C(x,x)$ is 1 if $x\in C$, and $\Delta_C(x,x)$ is $0$ otherwise.
Let $S=[S_0,S_1,\ldots,S_d]$ be an orthogonal matrix that diagonalizes the Bose--Mesner algebra, 
where $E_i=\frac{1}{|X|}S_iS_i^T$ for $i\in \{0,1,\ldots,d\}$.
We then define the $i$-th characteristic matrix $H_i$ of a subset $C$ of $X$ as $H_i=\Delta_CS_i$ for $i\in \{0,1,\ldots,d\}$.
We define $F_k=\frac{1}{|X|}H_iH_i^T$ for $i\in \{0,1,\ldots,d\}$.
Then $F_i$ coincides with the submatrix of $E_i$ obtained by restricting row and column indices to $C$.

We define a parameter of subsets in $Q$-polynomial schemes as shown below.
\begin{definition}
Let $C$ be a non-empty subset of a $d$-class $Q$-polynomial scheme $X$.
For $0\leq w^*\leq d-1$ and $1\leq t^*\leq d-w^*$, $C$ has a dual zero interval $\{w^*+1,\ldots,w^*+t^*\}$ if $b_{w^*+1}=\cdots=b_{w^*+t^*}=0$.
\end{definition}
If $(X,\mathcal{R})$ is a bipartite $Q$-polynomial scheme with dual eigenvalues $\theta_0^*>\cdots>\theta_d^*$ and $C$ is $\{x,y\}$ with $(x,y)\in R_d$, 
then the dual inner distribution of $C$ satisfies that $b_i$ is $0$ if $i$ is odd, and $m_i$ otherwise, where $m_i=\textrm{rank}E_i$.
Therefore a dual zero interval of $C$ is not necessarily uniquely determined.

For a subset $C$ having a dual zero interval $\{w^*+1,\ldots,w^*+t^*\}$, without loss of generality, we may assume $b_{w^*}\neq 0$ and one of the following:
\begin{enumerate}
\item $w^*+t^*+1\leq d$ and $b_{w^*+t^*+1}\neq0$,
\item $w^*+t^*=d$.
\end{enumerate} 
The dual zero interval $\{1,\ldots,t^*\}$  with $b_{t^*+1}\neq0$ is equivalent to the design, 
and the dual zero interval $\{w^*+1,\ldots,d\}$ with $b_{w^*}\neq 0$ is equivalent to the dual width.

Let $||\ ||$ stand for the Hermitian norm. 
\begin{lemma}\label{Q0}
Let $\mathbf{b}$ be the dual inner distribution of a subset $C$ of $X$.
Then for $i,j\in\{0,1,\ldots,d\}$ the characteristic matrices satisfy
$$||H_i^TH_j||^2=|C|\sum_{k=|i-j|}^{\min \{d,i+j\}}q_{i,j}^kb_k.$$
\end{lemma}
\begin{proof}
\begin{align*}
||H_i^TH_j||^2&=\text{Tr}(H_i^TH_jH_j^TH_i)=\text{Tr}(H_iH_i^TH_jH_j^T)=\text{Tr}(F_iF_j)\\
&=\sum\limits_{x\in C}F_iF_j(x,x)=\sum\limits_{x,y\in C}F_i(x,y)F_j(x,y)=\sum\limits_{x,y\in C}(E_i\circ E_j)(x,y)\\
&=\sum\limits_{x,y\in C}\sum\limits_{k=|i-j|}^{\min \{d,i+j\}}|X|q_{i,j}^kE_k(x,y)=|C|\sum_{k=|i-j|}^{\min \{d,i+j\}}q_{i,j}^kb_k.
\end{align*}
\end{proof}

\begin{proposition}\label{Q1}
Let $C$ be a non-empty subset of a $d$-class $Q$-polynomial scheme $(X,\mathcal{R})$.
\begin{enumerate}
\item The following are equivalent:
\begin{enumerate}
\item $b_k=0 \text{ for all } w^*+1\leq k\leq w^*+t^*$,
\item $H_i^TH_j=0 \text{ for } w^*+1\leq |i-j|, i+j\leq w^*+t^* $,
\item $H_k^TH_0=0 \text{ for } w^*+1\leq k\leq w^*+t^*$,
\item $(F_i,F_j)=0 \text{ for } w^*+1\leq |i-j|, i+j\leq w^*+t^*$,
\item $F_iF_j=0 \text{ for } w^*+1\leq |i-j|, i+j\leq w^*+t^*$.
\end{enumerate}
\item The following are equivalent:
\begin{enumerate}
\item $b_{w^*}>0$,
\item $H_i^TH_j\neq0 \text{ for } |i-j|=w^*$,
\item $H_{w^*}^TH_0\neq0$,
\item $(F_i,F_j)\neq0 \text{ for } |i-j|=w^*$,
\item $F_iF_j\neq0 \text{ for } |i-j|=w^*$
\end{enumerate}
\end{enumerate}
\end{proposition}
\begin{proof}
(1):
(a)$\Rightarrow$(b);
For $i,j$ satisfying $w^*+1\leq |i-j|, i+j\leq w^*+t^*$,
by the assumption of (a) and Proposition~\ref{Q0}, 
$||H_i^TH_j||^2=0$ holds. This implies that $H_i^TH_j=0$ holds.
Therefore (b) holds.

(b)$\Rightarrow$(c);
Setting $j=0$, (c) holds.

(c)$\Rightarrow$(a);
For $k$ satisfying $1\leq k\leq t^*$, 
$b_{w^*+k}=||H_{w^*+k}^TH_0||^2=0$.
Therefore (a) holds.

(b)$\Leftrightarrow$(d);
Since $F_i$ is a symmetric matrix for each $i\in \{0,1,\ldots,d\}$ and $||H_i^TH_j||^2=\text{Tr}(F_iF_j)=(F_i,F_j)$, (b) is equivalent to (d).

(b)$\Leftrightarrow$(e);
Since $H_i^TH_j=0$ for $w^*+1\leq |i-j|, i+j\leq w^*+t^*$, $F_iF_j=\frac{1}{|X|}H_iH_i^TH_jH_j^T=0$.
Therefore (b) implies (e).
Since $||H_i^TH_j||^2=\text{Tr}(F_iF_j)$, (e) implies (b). 

The proof of (2) is similar to that of (1).
\end{proof}
The following proposition gives a relation between an annihilator polynomial and the dual inner distribution of a subset $C$.
\begin{proposition}\label{Q-car}
Let $C$ be a non-empty subset of a $d$-class $Q$-polynomial scheme $X$
Let $F(x)$ be an annihilator polynomial of $C$.
We assume that there exist $\{f_i\}_{i=0}^d$ such that, for any $\theta^*\in \{\theta_0^*,\ldots,\theta_d^*\}$, $F(\theta^*)=\sum_{k=0}^df_kv_k^*(\theta^*)$, where $\{v_i^*(x)\}_{i=0}^d$ are the orthogonal polynomials corresponding to the second eigenmatrix.
Then $F(\theta_0^*)=\sum_{k=0}^df_kb_k$.
\end{proposition}
\begin{proof}
Since $F(\theta_i^*)=\sum_{k=0}^df_kv_k^*(\theta_i^*)$ for any $1\leq i\leq d$,  
\begin{align*}
\sum_{i=0}^dF(\theta_i^*)A_i&=\sum_{i=0}^d(\sum_{k=0}^df_kv_k^*(\theta_i^*))A_i \displaybreak[0]\\
&=\sum_{k=0}^df_k\sum_{i=0}^dv_k^*(\theta_i^*)A_i \displaybreak[0]\\
&=|X|\sum_{k=0}^df_kE_k.
\end{align*} 
Therefore $\sum_{i=0}^dF(\theta_i^*)\chi^TA_i\chi=|X|\sum_{k=0}^df_k\chi^TE_k\chi$ holds.
Since $F(x)$ is an annihilator polynomial of $C$, $\chi^TA_0\chi=|C|$ and since $\chi^TE_k\chi=\frac{|C|}{|X|}b_k$, we obtain the desired result.
\end{proof}
The following proposition holds, as in the case of strength of designs and dual width.
Proposition~\ref{Q-bound} (2) is already obtained in \cite{BGKM}, but we give the another proof.
\begin{proposition}\label{Q-bound}
Let $C$ be a non-empty subset of a $d$-class $Q$-polynomial scheme $X$ having a dual zero interval $\{w^*+1,\ldots,w^*+z^*\}$ with $b_{w^*}\neq0$ and degree $s$.
\begin{enumerate}
\item If $w^*+z^*+1\leq d$ and $b_{w^*+z^*+1}\neq0$, then $z^*\leq 2s$ holds.  
\item If $w^*+z^*=d$, then $z^*\leq s$ holds.
\end{enumerate}
\end{proposition} 
\begin{proof}
Replacing $w$, $z$, $s^*$, $S^*(C)$, the inner distribution by $w^*$, $z^*$, $s$, $S(C)$, the dual inner distribution respectively and using Proposition~\ref{Q-bound}, 
we have the desired results as the same method of the proof Proposition~\ref{P-bound}.
\end{proof}
The following is an algebraic proof of Proposition~\ref{Q-bound} using the Terwilliger algebra.
\begin{proof}
(Second proof of Proposition~\ref{Q-bound})

We denote $V=\mathbb{C}^{|X|}$.
Fix a base point $x\in C$. Let $A_i^*=A_i^*(x)$ be the diagonal matrix with $A_i^*(y,y)=|X|E_i(x,y)$ for each $i\in \{0,1,\ldots,d\}$.
Clearly $E_iV\cap E_jV=0$ for distinct integers $i,j\in \{0,1,\ldots,d\}$.
We define an $|X| \times (d+1)$ matrix $B^*=B^*(x)$ as 
$$B^*=(A_0^*\chi,A_1^*\chi,\ldots,A_d^*\chi).$$
Then, since $B^*P=|X|(E_0^*,E_1^*\chi,\ldots,E_d^*\chi)$, $\mathrm{rank}(B^*)=s_x+1$ where $s_x=|\{j\mid 1\leq j\leq d, E_j^*\chi\neq0\}|$.
Clearly $s_x\leq s$ holds.

(1);
Assume $w^*+t^*+1\leq d$, $b_{w^*+t^*+1}\neq0$.
Therefore $\chi|_{E_{w^*}V}\neq0 $, $\chi|_{E_{w^*+1}V}=\cdots=\chi|_{E_{w^*+t^*}V}=0 $ and $\chi|_{E_{w^*+t^*+1}V}\neq0$ hold. 
Consequently, $(A_1^*)^i\chi|_{E_{w^*+i}V}\neq0 $, $\chi|_{E_{w^*+i+1}V}=\cdots=\chi|_{E_{w^*+t^*-i}V}=0 $ and $(A_1^*)^i\chi|_{E_{w^*+t^*-i+1}V}\neq0$ hold for each $i$.
Therefore, since $w^*+i<w^*+t^*-i+1$ holds if and only if $i\leq \lfloor t^*/2\rfloor$, 
$\chi,A_1^*\chi,\ldots,(A_1^*)^{\lfloor t^*/2\rfloor}\chi$ are linearly independent.
Therefore, $\lfloor t^*/2\rfloor+1\leq \mathrm{rank}(B^*)=s_x+1\leq s+1$, i.e., $t^*\leq 2s$ holds.

(2); 
Assume $w^*+t^*=d$.
Therefore $\chi|_{E_{w^*}V}\neq0 $, $\chi|_{E_{w^*+1}V}=\cdots=\chi|_{E_{d}V}=0 $  hold. 
Consequently, $(A_1^*)^i\chi|_{E_{w^*+i}V}\neq0 $, $\chi|_{E_{w^*+i+1}V}=\cdots=\chi|_{E_{d}V}=0 $ hold for each $i$.
Therefore 
$\chi,A_1^*\chi,\ldots,(A_1^*)^{t^*}\chi$ are linearly independent.
Therefore, $t^*+1\leq \mathrm{rank}(B^*)=s_x+1\leq s+1$, i.e., $t^*\leq s$ holds.
\end{proof}
The following is the main theorem in this section.
The proof is a slight generalization of the proof presented by \cite[Theorem~2]{BGKM}.
\begin{theorem}\label{Q-mainth}
Let $C$ be a non-empty subset of a $d$-class $Q$-polynomial scheme $X$ having degree $s$.
Assume that there exists $w^*\in \{0,1,\ldots,d-s\}$ such that $b_{w^*}>0$ and that
for $(k,l)\in\{(x,y)\in \{0,1\ldots,w^*+s\}^2 \mid w^*+1\leq |x-y|\}$, 
$$ F_kF_l=0.$$
Then $(C,\mathcal{R}^C)$ is an $s$-class $Q$-polynomial scheme.
\end{theorem}
\begin{proof}
This proof is based almost completely on the proof of \cite[Theorem 2]{BGKM}.

Step 1: The set $\mathcal{F}_j=\{F_0,\ldots,F_{j-1},F_{w^*+j},\ldots,F_{w^*+s}\}$ is a basis for $\mathcal{A}$ for $0\leq j\leq s+1$.
\begin{proof}
When $j=s+1$, since $\theta_0^*,\ldots,\theta_s^*$ are mutually distinct and
\begin{align*}
(F_0,F_1,\ldots,F_s)=(\Delta_CA_0\Delta_C,\Delta_CA_1\Delta_C,\ldots,\Delta_CA_s\Delta_C) \begin{pmatrix}
1&v_1^*(\theta_0^*)&\cdots&v_s^*(\theta_0^*)\\
1&v_1^*(\theta_1^*)&\cdots&v_s^*(\theta_1^*)\\
\vdots&\vdots&\ddots&\vdots\\
1&v_1^*(\theta_s^*)&\cdots&v_s^*(\theta_s^*)
\end{pmatrix},
\end{align*}
the assertion holds for $j=s+1$.

When $j=s$, for the positive definite inner product $(M,N)=\text{tr}M^TN$, 
$F_{w^*+s}$ and $\{F_0,\ldots,F_{s-1}\}$ are orthogonal by Lemma~\ref{Q1}.
Therefore the assertion holds for $j=s$. 

Assume that the assertion is true for some $j>0$.
Since $\{F_0,\ldots,F_{j-2},F_{w^*+j},\ldots,F_{w^*+s}\}$ is linearly independent, 
it is sufficient to show that $F_{w^*+j-1}$ is linearly independent from $\{F_0,\ldots,F_{j-2},F_{w^*+j},\ldots,F_{w^*+s}\}$.
Assume $F_{w^*+j-1}=\sum\limits_{k=0}^{j-2}\alpha_kF_k+\sum\limits_{k=w^*+j}^{w^*+s}\alpha_kF_k$ for some $\alpha_k\in \mathbb{R}$. 
For $0\leq l\leq j-2$, by the assumption that
\begin{align*}
F_l\sum\limits_{k=0}^{j-2}\alpha_kF_k&=F_l\Big(F_{w^*+j-1}-\sum\limits_{k=w^*+j}^{w^*+s}\alpha_kF_k\Big)\\
&=0.
\end{align*}
Therefore $\Big(\sum\limits_{k=0}^{j-2}\alpha_kF_k\Big)^2=0$.
Since $\sum\limits_{k=0}^{j-2}\alpha_kF_k$ is a real symmetric matrix, $\sum\limits_{k=0}^{j-2}\alpha_kF_k=0$.
Since $\{F_0,\ldots,F_{j-2}\}$ is linearly independent, $\alpha_k=0$ for $0\leq k\leq j-2$.
Therefore, we obtain $F_{w^*+j-1}=\sum\limits_{k=w^*+j}^{w^*+s}\alpha_kF_k$.
Multiplying $F_{j-1}$, we obtain $F_{j-1}F_{w^*+j-1}=0$, which is a contradiction by Proposition~\ref{Q1}.

Therefore $F_{w^*+j-1}$ is linearly independent from $\{F_0,\ldots,F_{j-2},F_{w^*+j},\ldots,F_{w^*+s}\}$, and the assertion holds for $j-1$.
\end{proof} 
Define $\bar{\mathcal{E}_j}=\mathrm{Span}\{F_0,F_1,\ldots,F_j\}$ and $\bar{\mathcal{A}}=\{\bar{M}\mid M\in \mathcal{A}\}$ where $\bar{M}$ is the submatrix of $M$ obtained by restricting row and column to $C$.
 
Step 2: 
\begin{enumerate}
\item The set $\mathcal{F}_j'=\{F_0,\ldots,F_{j-1},I,F_{w^*+j+1}\ldots,F_{w^*+s}\}$ is a basis for $\mathcal{A}$ for $0\leq j\leq s$.
\item $\bar{\mathcal{A}}\bar{\mathcal{E}_j}=\bar{\mathcal{E}_j}\bar{\mathcal{A}}=\bar{\mathcal{E}_j}$ for $0\leq j\leq s$.
\end{enumerate}
\begin{proof}
First $\bar{I}\in \bar{\mathcal{A}}$ holds.
When $j=0$, assume $\bar{I}$ is linearly dependent from $\{F_{w^*+1},\ldots,F_{w^*+s}\}$.
Write $\bar{I}=\sum_{i=1}^s\alpha_{w^*+i}F_{w^*+i}$.
Then multiplying $F_0$ gives $F_0=0$, which is a contradictiton.
Therefore $\mathcal{F}_0'$ is a basis of $\bar{\mathcal{A}}$. 
In order to verify (2) for $j=0$, it is enough to show that $\bar{\mathcal{E}_0}$ is closed under the multiplication of any element of a basis $\mathcal{F}_0'$. 
It follows from the definition that $F_0F_{w^*+i}=F_{w^*+i}F_0=0$ for any $i\in \{1,\ldots,s\}$, and clearly $\bar{I}F_0=F_0\bar{I}=F_0$. 
Therefore the assertion holds (2) for $j=0$.

Assume that (1) and (2) true for some $s>j>0$.
Assume that $\bar{I}$ is linearly dependent from $F_0,\ldots,F_{j-1},F_{w^*+j+1},\ldots,F_{w^*+s}\}$.
Write $\bar{I}=\sum_{i=0}^{j-1}\alpha_{i}F_{i}+\sum_{i=j+1}^s\alpha_{w^*+i}F_{w^*+i}$.
Then multiplying $F_j$ gives $F_j\in\bar{\mathcal{A}}\bar{\mathcal{E}_{j-1}}=\bar{\mathcal{E}_{j-1}}$, which contradicts that $\mathcal{F}_0$ is a basis.
Therefore $\mathcal{F}_j'$ is a basis of $\bar{\mathcal{A}}$. 
Then
\begin{align*}
\bar{\mathcal{A}}\bar{\mathcal{E}_j}&=\bar{\mathcal{A}}\bar{\mathcal{E}_{j-1}}+\bar{\mathcal{A}}\text{Span}\{F_j\}  \displaybreak[0]\\
&=\bar{\mathcal{E}}_{j-1}+\text{Span}\mathcal{F}_j'\text{Span}\{F_j\} \displaybreak[0]\\
&=\bar{\mathcal{E}}_{j-1}+\text{Span}\{F_0,\ldots,F_{j-1},\bar{I}\}\text{Span}\{F_j\} \displaybreak[0]\\
&=\bar{\mathcal{E}}_{j-1}+\text{Span}\{F_j\} \displaybreak[0]\\
&=\bar{\mathcal{E}_j}.
\end{align*}
$\bar{\mathcal{E}_j}\bar{\mathcal{A}}=\bar{\mathcal{E}_j}$ is also shown by similar method.
Therefore the assertion holds (2) for $j$.
\end{proof}
Therefore $(C,\mathcal{R}^C)$ is an $s$-class symmetric association scheme since $\bar{A}_i$ is a symmetric matrix for $i\in S(C)\cup\{0\}$.
Finally we show that this scheme is $Q$-polynomial.
Clearly, $\mathrm{Span}\{\bar{J}\}=\bar{\mathcal{E}_0}\subsetneq \bar{\mathcal{E}_1}\subsetneq\ldots\subsetneq \bar{\mathcal{E}_s}=\bar{\mathcal{A}}$ holds and Step~2 (2) implies that $\bar{\mathcal{E}_j}$ is an ideal of the adjacency algebra .
Generally, an ideal of the adjacency algebra is spanned by some primitive idempotents.
Therefore we can take primitive idempotents $E_0,E_1,\ldots,E_s$ of $(C,\mathcal{R}^C)$ such that 
$E_i\in \bar{\mathcal{E}_i}\setminus \bar{\mathcal{E}_{i-1}}$ for $1\leq i\leq s$.
Since, for each $i$, the ideal $\bar{\mathcal{E}_i}$ is spanned by all polynomials of $F_1$ as entrywise product at most degree $i$, 
the primitive idempotent $E_i$ is a polynomial of $E_1$ as an entrywise product.
Therefore $(C,\mathcal{R}^C)$ is $Q$-polynomial.
\end{proof}

\begin{corollary}{\upshape \cite[Theorem~2]{BGKM}}
Let $C$ be a non-empty subset of a $d$-class $Q$-polynomial scheme $X$ having dual width $w^*$ and degree $s$.
If $w^*+s=d$, then $(C,\mathcal{R}^C)$ is an $s$-class $Q$-polynomial scheme.
\end{corollary}
\begin{proof}
It is sufficient to verify the assumption Theorem~\ref{Q-mainth} for dual width $w^*$.
For $w^*<|k-l|$, by Lemma~\ref{Q0} 
$$||H_k^TH_l||^2=|C|\sum\limits_{h=0}^dq_{k,l}^jb_k =|C|\Big(\sum\limits_{h=0}^{|k-l|}q_{k,l}^hb_h+\sum\limits_{h=|k-l|}^dq_{k,l}^hb_h \Big).$$
Since $b_h=0$ when $|k-l|\leq h$, $q_{k,l}^h=0$ when $h<|k-l|$, $||H_k^TH_l||=0$ i.e., $H_k^TH_l=0$ holds.
Therefore $F_kF_l=0$ for $w^*<|k-l|$.
\end{proof}

\begin{corollary}\label{Q-int}
Let $C$ be a non-empty subset of a $d$-class $Q$-polynomial scheme $X$ having a dual zero interval $\{w^*+1,\ldots,w^*+t^*\}$ and degree $s$.
If $2s-1\leq t^*$, then $(C,\mathcal{R}^C)$ is an $s$-class $Q$-polynomial scheme.
\end{corollary}
\begin{proof} 
It is sufficient to verify the assumption of Theorem~\ref{Q-mainth} for $w^*$, which appears in a dual zero interval.
By Proposition~ \ref{Q1} (2), $H_k^TH_l=0$ for $w^*+1\leq |k-l|, k+l\leq w^*+2s-1$.
Therefore $F_kF_l=0$ for $w^*+1\leq |k-l|, k+l\leq w^*+2s-1$, in particular for $(k,l)\in\{(x,y)\in \{0,1,\ldots,w^*+s\}^2\mid |x-y|\geq w^*+1\}$.
\end{proof}
Corollary~\ref{Q-int} gives a sufficient condition that a subset having a nice dual zero interval carries a $Q$-polynomial scheme.
When $w^*=0$, Corollary~\ref{Q-int} implies that a design with maximum strength $t^*$ and degree $s$ satisfying $2s-1\leq t^*$ carries a $Q$-polynomial scheme.
However, in fact, it is well-known in \cite[Theorem~5.25]{D} that the assumption $2s-2\leq t^*$ engenders an identical result, i.e., a design with maximum strength $z^*$ and degree $s$ satisfying $2s-2\leq t^*$ carries a $Q$-polynomial scheme.

\begin{example}\label{Q-example}
Since each $C$ appearing in Example~\ref{P-example} is a linear code, we consider the dual code $C^*$.
Then $C^*$ has the inner distribution $\frac{1}{|C|}\mathbf{b}$ and the dual inner distribution $\frac{|X|}{|C|}\mathbf{a}$. 
\begin{enumerate}
\item Let $C^*$ be the dual code of $C$ appearing in Example~\ref{P-example}~(1).
$C^*$ has the dual zero interval $\{2,3,\ldots,2n-2\}$ and degree set $S(C)=\{2k\mid 2\leq k \leq n-1\}$. 
Therefore $C^*$ satisfies $t^*= 2s+1$, which implies that $C^*$ is a $Q$-polynomial scheme.
\item Let $C^*$ be the dual code of $C$ appearing in Example~\ref{P-example}~(2).
Then $C^*$ has a dual zero interval $\{w^*+1,\ldots,w^*+t^*\}$ with $w^*=2^m-4$ and $t^*=2$ and the degree $s=1$.
Therefore $C^*$ satisfies $t^*=2s$, which implies that $C^*$ is a $Q$-polynomial scheme.
The strength of $C^*$ is $3$, as already shown \cite[Theorem~5.25]{D}.
\item Let $C^*$ be the dual code of $C$ appearing in Example~\ref{P-example}~(3).
Then $C^*$ has a dual zero interval $\{w^*+1,\ldots,w^*+t^*\}$ with $w^*=2^m-4$ and $t^*=3$ and the degree $s=2$.
Therefore $C^*$ satisfies $t^*=2s-1$, which implies that $C$ is a $Q$-polynomial scheme.
Note the strength of $C^*$ is $4$, as already shown in \cite[Theorem~5.25]{D}.
\item Let $C^*$ be the dual code of $C$ appearing in Example~\ref{P-example}~(4).
Then $C^*$ has a dual zero interval $\{w^*+1,\ldots,w^*+t^*-1\}$ with $w^*=16$ and $t^*=6$ and degree $s=3$.
Therefore $C^*$ satisfies $t^*=2s$, which implies that $C$ is a $Q$-polynomial scheme.
Note the strength of $C^*$ is $7$. This is already shown \cite[Theorem~5.25]{D}.
\item Let $C^*$ be the dual code of $C$ appearing in Example~\ref{P-example}~(5). 
Then $C^*$ has dual zero intervals $\{w^*+1,\ldots,w^*+t^*\}$ with $(w^*,t^*)=(1,5)$ or $(w,t)=(17,5)$ and the degree $s=3$.
Therefore $C^*$ satisfies $t^*=2s-1$, which implies that $C$ is a $Q$-polynomial scheme.
\item Let $C^*$ be the dual code of $C$ appearing in Example~\ref{P-example}~(6).
Then $C^*$ has a dual zero interval $\{w^*+1,\ldots,w^*+t^*\}$ with $w^*=16$ and $t^*=7$ and the degree $s=4$.
Therefore $C^*$ satisfies $t^*=2s-1$, which implies that $C$ is a $Q$-polynomial scheme.
Note the strength of $C^*$ is $8$. This is already shown in \cite[Theorem~5.25]{D}.
\end{enumerate}
\end{example}
\section{Spherical analogue of the dual zero interval}
In this section, we will discuss the spherical analogue of the dual zero interval as spherical designs.
First, we define spherical designs and characterize spherical designs using Gegenbauer polynomials and spherical characteristic matrices. 

For a positive integer $t$, a finite non-empty set $X$ in the unit sphere $S^{d-1}$ is called 
a spherical $t$-design in $S^{d-1}$ if the following condition is satisfied for all polynomials $f(x)=f(x_1,\dots,x_d)$ of degree not exceeding $t$:
$$\frac{1}{|S^{d-1}|}\int\nolimits_{S^{d-1}}f(x)d\sigma(x)=\frac{1}{|X|}\sum\limits_{x \in X}f(x).$$
Here, $\sigma$ is the Haar measure on $S^{d-1}$; $|S^{d-1}|$ denotes the volume of the sphere $S^{d-1}$.

We define Gegenbauer polynomials $\{Q_k(x)\}_{k=0}^\infty$ on $S^{d-1}$ as
\begin{align*}
& Q_0(x)=1,\quad Q_1(x)=dx,\\
& \frac{k+1}{d+2k}Q_{k+1}(x)=xQ_k(x)-\frac{d+k-3}{d+2k-4}Q_{k-1}(x).
\end{align*}
Using three-term recurrence formula of Gegenbauer polynomials, we can define non-negative numbers $q_k(i,j)$ satisfying the following equation for non-negative integers $i,j,k$: 
\begin{align*}
Q_i(x)Q_j(x)=\sum_{k=|i-j|}^{i+j}q_k(i,j)Q_k(x).
\end{align*} 
It is well-known that $q_k(i,j)>0$ if and only if $|i-j|\leq k\leq i+j$ and $k\equiv i+j\pmod{2}$.
We define $b_k=\frac{1}{|X|}\sum\limits_{x,y\in X}Q_k(\langle x,y\rangle)$ for each positive integer $k$.

Let $\mbox{Hom}(\mathbb{R}^d)$ be the vector space of the polynomials over $\mathbb{R}$, 
and let $\mbox{Hom}_l(\mathbb{R}^d)$ be the subspace of $\mbox{Hom}(\mathbb{R}^d)$ consisting of polynomials of total degree at most $l$. 
Let $\mbox{Harm}(\mathbb{R}^d)$ be the vector space of the harmonic polynomials over $\mathbb{R}$
and $\mbox{Harm}_l(\mathbb{R}^d)$ be the subspace of $\mbox{Harm}(\mathbb{R}^d)$ consisting of homogeneous polynomials of total degree $l$.
The direct sum decomposition of $\mbox{Hom}_l(\mathbb{R}^d)$ is known, as
$$\mbox{Hom}_l(\mathbb{R}^d)=\bigoplus_{k=0}^{\lfloor l/2\rfloor}(x_1^2+\cdots+x_d^2)^k\mbox{Harm}_{l-2k}(\mathbb{R}^d).$$ 
Let $\{\phi_{l,1},\dots,\phi_{l,h_l}\}$ be an orthonormal basis
of $\text{Harm}_l(\mathbb{R}^d)$ with respect to the inner product 
$$\langle\phi,\psi \rangle=\frac{1}{|S^{d-1}|}\int\nolimits_{S^{d-1}}\phi(x)\psi(x)d\sigma(x).$$
Then the addition formula for the Gegenbauer polynomial holds \cite[Theorem 3.3]{DGS}, as
\begin{lemma}\label{add}
$\sum\limits_{i=1}^{h_l}\phi_{l,i}(x)\phi_{l,i}(y)=Q_l(\langle x,y\rangle)$ for any $l\in \mathbb{N}$, $x,y\in S^{d-1}$.
\end{lemma}
We define the $l$-th spherical characteristic matrix of a finite set $X\subset S^{d-1}$ as the $|X| \times h_l$ matrix
$$ H_l=(\phi_{l,i}(x))_{\substack{x\in X\\1\leq i\leq h_l}} .$$

\begin{lemma}\label{S0}
Letting $X$ be a finite non-empty set in $S^{d-1}$, and letting $t$ be a positive integer,
for non-negative integers $i,j$, the spherical characteristic matrices satisfy
$$||H_i^TH_j||^2=|X|\sum_{k=|i-j|}^{i+j}q_k(i,j)b_k.$$
\end{lemma}
\begin{proof}
See Lemma~\ref{Q0}.
\end{proof}

A criterion for spherical $t$-designs using Gegenbauer polynomials and the spherical characteristic matrices is known \cite[Theorem 5.3, 5.5]{DGS}.
\begin{lemma}\label{S1}
Let $X$ be a non-empty finite set in $S^{d-1}$. The following conditions are equivalent:
\begin{enumerate}
\item $X$ is a spherical $t$-design,
\item $b_k=0$ for any $k\in \{1,\ldots,t\}$, and
\item $H_k^TH_l=\delta_{k,l}|X|I \quad \text{for} \quad 0\leq k+l\leq t $. 
\end{enumerate}
\end{lemma}

Next, we define a generalization of spherical design and spherical dual zero interval.
For positive integers $w$ and $t$, a finite non-empty set $X$ in the unit sphere $S^{d-1}$ is called 
a spherical $(w,t)$-design in $S^{d-1}$ if the following condition is satisfied:
$$\frac{1}{|S^{d-1}|}\int\nolimits_{S^{d-1}}f(x)d\sigma(x)=\frac{1}{|X|}\sum\limits_{x \in X}f(x)$$
for all polynomials $f(x)=f(x_1,\dots,x_d)\in \bigoplus_{l=1}^t\bigoplus_{k=0}^{\lfloor l/2\rfloor}(x_1^2+\cdots+x_d^2)^k\mbox{Harm}_{w+l-2k}(\mathbb{R}^d)$.

A spherical $(0,t)$-design coincides with a spherical $t$-design.
Seymour and Zaslavsky showed the following existence theorem: 
\begin{theorem}{\upshape \cite[Main Theorem]{SZ}}
Let $\Omega$ be a path-connected topological space provided with a positive finite measure $\mu$ that satisfies $\mu(S)\geq0$ for any measurable set $S$ and $\mu(U)>0$ for any no-empty open set.
Let $f_i:\Omega\rightarrow \mathbb{R}^p$ be continuous integrable functions for $i\in \{1,\ldots,m\}$.
Then a finite set $X$ of $\Omega$ exists that satisfies 
$$\frac{1}{\mu(\Omega)}\int\nolimits_{\Omega}f_i(x)d\mu(x)=\frac{1}{|X|}\sum\limits_{x \in X}f_i(x)\text{ for }i\in \{1,\ldots,m\}.$$
\end{theorem}
Setting $p=1$, $\Omega=S^{d-1}$, $\mu=\sigma$ and $\{f_1,\ldots,f_m\}$ is a basis of the linear space $\bigoplus_{l=1}^t\bigoplus_{k=0}^{\lfloor l/2\rfloor}(x_1^2+\cdots+x_d^2)^k\mbox{Harm}_{w+l-2k}(\mathbb{R}^d)$; we can obtain the following existence theorem for spherical $(w,t)$-designs:
\begin{corollary}
For any positive integers $w,t,d$, a spherical $(w,t)$ design exists in $S^{d-1}$.
\end{corollary} 

As a spherical design, we can characterize spherical $(w,t)$-designs by Gegenbauer polynomials and the spherical characteristic matrices as follows:
\begin{lemma}\label{S2}
Let $X$ be a finite non-empty set in $S^{d-1}$.
\begin{enumerate}
\item The following are equivalent:
\begin{enumerate}
\item $X$ is a spherical $(w,t)$-design,
\item $b_k=0$ for any $k\in \{w+1,\ldots,w+t\}$,
\item $H_i^TH_j=0 \quad \text{for} \quad w+1\leq |i-j|, i+j\leq w+t$,
\item $H_k^TH_0=0 \text{ for } w+1\leq k\leq w+t $,
\item $F_iF_j=0 \text{ for } w+1\leq |i-j|, i+j\leq w+t $. 
\end{enumerate}
\item The following are equivalent:
\begin{enumerate}
\item $b_{w}>0$,
\item $H_i^TH_j\neq0 \text{ for } |i-j|=w$,
\item $H_w^TH_0\neq0$,
\item $F_iF_j\neq0 \text{ for } |i-j|=w$.
\end{enumerate}
\end{enumerate}
\end{lemma}
\begin{proof}
It follows from the proof of Lemma~\ref{Q1} that the equivalence among $(b),\ldots,(e)$ of (1) and among $(a),\ldots,(d)$ of (2), so we verify the equivalence between $(a)$ and $(d)$ of (1).
Since $\langle 1,f\rangle=\frac{1}{|S^{d-1}|}\int\nolimits_{S^{d-1}}f(x)d\sigma(x)$ for any $f\in \mbox{Hom}(\mathbb{R}^d)$ and $H_0$ is orthogonal to $H_i$ for any $i\geq 1$ with respect to the inner product, 
$X$ is a spherical $(w,t)$-design if and only if 
\begin{align}\label{S2-1}
\sum\limits_{x \in X}f(x)=0\text{ for any }f\in\bigoplus_{i=1}^t\mbox{Harm}_{w+i}(\mathbb{R}^d).
\end{align}  
Since $\{\phi_{l,1},\dots,\phi_{l,h_l}\}$ is a basis of $H_l$, (\ref{S2-1}) is equivalent to 
\begin{align}\label{S2-2}
\sum\limits_{x \in X}f(x)=0\text{ for any }f\in\{\phi_{l,1},\dots,\phi_{l,h_l}\}, \l\in \{w+1,\ldots,w+t\}.
\end{align}
Additionally, (\ref{S2-2}) is equivalent to $H_l^TH_0=0$ for any $\l\in \{w+1,\ldots,w+t\}$, which proves that $(a)$ of $(1)$ is equivalent to $(d)$ of $(1)$. 
\end{proof}
We define a spherical dual zero interval as follows.
\begin{definition}
Let $X$ be a non-empty finite set in $S^{d-1}$.
For non-negative integer $w$ and positive integer $t$, $X$ has a spherical dual zero interval $\{w+1,\ldots,w+t\}$ if $b_{w}\neq 0$ and $b_{w+1}=\cdots=b_{w+t}=0$.
\end{definition}
If $X=\{x,y\}$ with $\langle x,y\rangle=-1$, 
then $\{b_k\}_{k=0}^\infty$ satisfies that $b_i$ is $0$ if $i$ is odd, and $Q_i(1)>0$ otherwise.
Therefore, a spherical dual zero interval of $X$ is not necessarily uniquely determined.
A spherical $(w,t)$-design with $b_w>0$ has a spherical dual zero interval $\{w+1,\ldots,w+t\}$;
consequently, the spherical $t$-design coincides with a spherical dual zero interval $\{1,\ldots,t\}$.

We define the degree set $A(X)$ of a finite non-empty set $X$ in $S^{d-1}$ as 
$$A(X)=\{ \langle x, y\rangle \mid x,y \in X, x\neq y\},$$
and degree $s$ by the cardinality of the degree set $A(X)$.
Let $A(X)=\{\alpha_1,\ldots,\alpha_s\}$ and $\alpha_0=1$, and 
define $R_i=\{(x,y)\in X \times X\mid \langle x,y\rangle=\alpha_i\}$ for each $i\in \{0,1,\ldots,s\}$ and
$R^X=\{R_0,R_1,\ldots,R_s\}$.

A polynomial $F(x)$ is called an annihilator polynomial of a finite non-empty set $X$ in $S^{d-1}$ if 
\begin{align*}
F(\alpha)=0\ \text{ for any } \alpha\in A(X).
\end{align*}

In Proposition~\ref{Q-car},\ref{Q-bound},\ref{Q-mainth} and \ref{Q-int}, replacing $\{v_i^*(x)\}_{i=0}^d$, characteristic matrices, dual zero interval by $\{Q_i(x)\}_{i=0}^\infty$, spherical characteristic matrices, and the spherical dual zero interval respectively, we obtain the following proposition. 
The proofs are exactly the same.
\begin{proposition}\label{S-car}
Let $X$ be a finite non-empty set of $X$ in the unit sphere $S^{d-1}$.
Let $F(x)$ be an annihilator polynomial of $X$ with degree $l$ and let $\{f_i\}_{i=0}^l$ be the coefficients of $F(x)$ expressed in terms of the polynomials $\{Q_i(x)\}_{i=0}^l$.
Then $F(d)=\sum_{k=0}^lf_kb_k$.
\end{proposition}
\begin{proposition}\label{S-bound}
Let $X$ be a finite non-empty set in the unit sphere $S^{d-1}$ having a spherical dual zero interval $\{w+1,\ldots,w+t\}$ and degree $s$.
Then $t\leq 2s$ holds. 
\end{proposition}
\begin{theorem}\label{S-mainth}
Let $X$ be a finite non-empty set in the unit sphere $S^{d-1}$ having degree $s$.
Assume that there exists a nonnegative integer $w$ such that $b_{w}>0$ and 
for $(k,l)\in\{(x,y)\in \{0,1\ldots,w+s\}^2 \mid w+1\leq |x-y|\}$, 
$$ F_kF_l=0.$$
Then $(X,R^X)$ is an $s$-class $Q$-polynomial scheme.
\end{theorem}
\begin{corollary}\label{S-int}
Let $X$ be a finite non-empty set in the unit sphere $S^{d-1}$ having a spherical dual zero interval $\{w+1,\ldots,w+t\}$ and degree $s$.
If $2s-1\leq t$, then $(X,R^X)$ is an $s$-class $Q$-polynomial scheme.
\end{corollary}
\begin{problem}
Does there exist a finite non-empty set in the unit sphere $S^{d-1}$ satisfying the assumption of Corollary~\ref{S-int}?
\end{problem}
\section*{Acknowledgements}
The author would like to thank Professor Akihiro Munemasa, Hiroshi Suzuki and Hajime Tanaka  
for helpful discussions.
This work is supported by Grant-in-Aid for JSPS Fellows.

\end{document}